\documentclass[11pt, reqno]{my-jnsa-t}

\usepackage{amsmath, amsthm, amscd, amsfonts, amssymb, graphicx, color}
\RequirePackage[backref,colorlinks=true]{hyperref}

\begin{document}

\title{On the Henstock--Kurzweil Integral for Riesz-space-valued Functions on Time Scales}

\author[a1,a2]{Xuexiao You}
\ead{youxuexiao@126.com}

\author[a1,a3]{Dafang Zhao}
\ead{dafangzhao@163.com}

\author[a4]{Delfim F. M. Torres\corref{c1}}
\ead{delfim@ua.pt}


\address[a1]{School of Mathematics and Statistics,
Hubei Normal University, Huangshi,
Hubei 435002, P. R. China
\vskip 0.1cm}

\address[a2]{College of Computer and Information,
Hohai University, Nanjing, Jiangsu 210098,
P. R. China
\vskip 0.1cm}

\address[a3]{College of Science,
Hohai University, Nanjing, Jiangsu 210098,
P. R. China
\vskip 0.1cm}

\address[a4]{Center for Research and Development
in Mathematics and Applications (CIDMA),\\
Department of Mathematics, University of Aveiro,
3810-193 Aveiro, Portugal.}


\cortext[c1]{Corresponding author}
\setcounter{page}{1}

\recivedat{\ }


\authors{X. You, D. Zhao, D. F. M. Torres}


\begin{abstract}
We introduce and investigate the Henstock--Kurzweil (HK)
integral for Riesz-space-valued functions on time scales.
Some basic properties of the HK delta integral
for Riesz-space-valued functions are proved.
Further, we prove uniform and monotone
convergence theorems. 
\begin{keyword}Henstock--Kurzweil integral
\sep Riesz space
\sep time scales.
\MSC{28B05 \sep 28B10 \sep 28B15 \sep 46G10.}
\end{keyword}
\end{abstract}


\maketitle

\newtheorem{theorem}{Theorem}[section]
\newtheorem{lemma}[theorem]{Lemma}
\newtheorem{proposition}[theorem]{Proposition}
\newtheorem{corollary}[theorem]{Corollary}
\newtheorem{question}[theorem]{Question}

\theoremstyle{definition}
\newtheorem{definition}[theorem]{Definition}
\newtheorem{algorithm}[theorem]{Algorithm}
\newtheorem{conclusion}[theorem]{Conclusion}
\newtheorem{problem}[theorem]{Problem}

\theoremstyle{remark}
\newtheorem{remark}[theorem]{Remark}
\numberwithin{equation}{section}


\section{Introduction}

It is well known that the Henstock--Kurzweil integral integrates
highly oscillating functions and encompasses Newton, Riemann and
Lebesgue integrals. This integral was introduced by Kurzweil
and Henstock independently in 1957/58 \cite{H1,K1}. It has been shown
that the Henstock--Kurzweil integral is equivalent to the Denjoy--Perron integral.
For fundamental results and some applications in the theory of Henstock--Kurzweil integration,
we refer the reader to the papers \cite{BPM,BLMMP,C1,PM,H,HS,H1,S1,S2,Y,ZY}
and monographs \cite{B0,G,H2,H3,KS,K2,K3,L,L1,L2,LV,P,S,SY}.

A time scale $\mathbb{T}$ is an arbitrary nonempty closed subset of real numbers
$\mathbb{R}$ with the subspace topology inherited from the standard topology
of $\mathbb{R}$. The theory of time scales was born in 1988
with the Ph.D. thesis of Hilger \cite{H4}. The aim of this theory
is to unify various definitions and results
from the theories of discrete and continuous dynamical systems,
and to extend such theories to more general classes of dynamical systems.
It has been extensively studied on various aspects by several
authors; see, e.g., \cite{BCT2,BCT1,BL,BP1,BP2,BS,G1,OTT}.
In \cite{PT}, Peterson and Thompson introduced a more general
concept of integral on time scales, i.e., the Henstock--Kurzweil delta integral,
which contains the Riemann delta and the Lebesgue delta integrals as special cases.
The theory of Henstock--Kurzweil integration for real-valued and vector-valued
functions on time scales has developed rather intensively in the past few years;
see, for instance, the papers \cite{A,C,FMS,MS2,MS1,ST1,N2,N1,S3,T,Y1} and the
references cited therein.

One of the interesting points of integration theories is the problem
when functions with values in general spaces have to be integrated.
The Henstock--Kurzweil integral for Riesz-space-valued functions
was investigated in \cite{B1,B2,BCR,BCS,BR1,BR2,BR4,BRV,BS1,BS2,BS3,BST,R1,R3,R4,R5,R6,R7}.
Surprisingly enough, the Henstock--Kurzweil integral for Riesz-space-valued functions
has not received attention in the literature of time scales. The main goal of this paper
is to generalize the results above by constructing the Henstock--Kurzweil integral
for Riesz-space-valued functions on time scales.

The paper is organized as follows.
Section~\ref{sec:2} contains basic concepts of Riesz space,
time scales and Henstock--Kurzweil integral. In Section~\ref{sec:3},
the definition of Henstock--Kurzweil delta integral for Riesz-space-valued
functions is introduced, and the basic properties of the Henstock--Kurzweil
delta integral for Riesz-space-valued functions are investigated.
In Section~\ref{sec:4}, we prove a uniformly convergence
theorem and a monotone convergence theorem for the Henstock--Kurzweil
delta integral for Riesz-space-valued functions.


\section{Preliminaries}
\label{sec:2}

The following conventions and notations will be used, unless stated otherwise.
Let $\mathbb{N}$, $\mathbb{R}$, and $\mathbb{R}^{+}$ be the sets of all natural,
real and positive real numbers, respectively, and let $X$ be a Riesz space.
A decreasing sequence $(b_{n})_{n}$ in $X$, such that $\bigwedge_{n}b_{n}=0$,
is called an $(o)$-sequence. A bounded double sequence $(a_{i,j})_{i,j}$
in $X$ is a $(D)$-sequence or a regulator if $(a_{i,j})_{i,j}$ is an $(o)$-sequence
for all $i\in \mathbb{N}$. A Riesz space $X$ is said to be Dedekind complete
if every nonempty subset $X_{1}$ of $X$, bounded from above, has a lattice
supremum in $X$ denoted by $\bigvee X_{1}$. A Dedekind complete Riesz space $X$
is said to be weakly $\sigma$-distributive if for every $(D)$-sequence
$(a_{i,j})_{i,j}$ in $X$ one has:
$$
\bigwedge_{\varphi \in\mathbb{N}^{\mathbb{N}}}\Big(\bigvee_{i=1}^{\infty}a_{i,\varphi(i)}\Big)=0.
$$
Let $\mathbb{T}$ be a time scale, i.e., a nonempty closed subset of $\mathbb{R}$.
For $a,b\in \mathbb{T}$, we define the closed interval
$[a,b]_{\mathbb{T}}$ by $[a,b]_{\mathbb{T}}=\{t\in\mathbb{T}: a \leq t \leq b\}$.
The open and half-open intervals are defined in an similar way.
For $t\in \mathbb{T}$, we define the forward jump operator $\sigma$
by $\sigma(t)=\inf\{s>t: s\in \mathbb{T}\}$, where
$\inf\emptyset=\sup\{\mathbb{T}\}$, while the
backward jump operator $\rho$ is defined by
$\rho(t)=\sup\{s<t: s\in \mathbb{T}\}$,
where $\sup\emptyset=\inf\{\mathbb{T}\}$.

If $\sigma(t)>t$, then we say that $t$ is right-scattered, while if
$\rho(t)<t$, then we say that $t$ is left-scattered. If $\sigma(t)=t$,
then we say that $t$ is right-dense, while if $\rho(t)=t$, then
we say that $t$ is left-dense. A point $t\in \mathbb{T}$ is dense
if it is right and left dense; isolated, if it is right and left scattered.
The forward graininess function $\mu(t)$ and the backward graininess function
$\eta(t)$ are defined by $\mu(t)=\sigma(t)-t$ and $\eta(t)=t-\rho(t)$
for all $t\in \mathbb{T}$, respectively. If $\sup\mathbb{T}$ is finite
and left-scattered, then we define $\mathbb{T}^{\kappa}:=\mathbb{T}\backslash \sup\mathbb{T}$,
otherwise $\mathbb{T}^{\kappa}:=\mathbb{T}$; if $\inf\mathbb{T}$ is finite and
right-scattered, then $\mathbb{T}_{\kappa}:=\mathbb{T}\backslash
\inf\mathbb{T}$, otherwise $\mathbb{T}_{\kappa}:=\mathbb{T}$. We set
$\mathbb{T}^{\kappa}_{\kappa}:=\mathbb{T}^{\kappa}\bigcap \mathbb{T}_{\kappa}$.

Throughout this paper, all considered intervals will be intervals in
$\mathbb{T}$. A partition $\mathcal{D}$ of $[a,b]_{\mathbb{T}}$ is a
finite collection of interval-point pairs
${\{([t_{i-1},t_{i}]_{\mathbb{T}},\xi_{i})\}}^{n}_{i=1}$, where
$$
\{a=t_{0}<t_{1}<\cdots <t_{n-1}<t_{n}=b\}
$$
and $\xi_{i}\in[a,b]_{\mathbb{T}}$ for $i=1,2,\cdots,n$.
By $\Delta t_{i}=t_{i}-t_{i-1}$ we denote the length
of the $i$th subinterval in the partition $\mathcal{D}$.
We say that $\delta(\xi)=(\delta_{L}(\xi),\delta_{R}(\xi))$
is $\Delta$-gauge for $[a,b]_{\mathbb{T}}$ provided $\delta_{L}(\xi)>0$
on $(a,b]_{\mathbb{T}}$, $\delta_{R}(\xi)>0$ on $[a,b)_{\mathbb{T}}$,
$\delta_{L}(a)\geq 0,\delta_{R}(b)\geq 0$ and $\delta_{R}(\xi)\geq
\mu(\xi)$ for all $\xi\in [a,b)_{\mathbb{T}}$.
Let $\delta^{1}(\xi), \delta^{2}(\xi)$ be $\Delta$-gauges for
$[a,b]_{\mathbb{T}}$ such that $0< \delta^{1}_{L}(\xi)< \delta^{2}_{L}(\xi)$
for all $\xi\in (a,b]_{\mathbb{T}}$ and
$0< \delta^{1}_{R}(\xi)< \delta^{2}_{R}(\xi)$
for all $\xi\in [a,b)_{\mathbb{T}}$. We say $\delta^{1}(\xi)$
is finer than $\delta^{2}(\xi)$ and write
$\delta^{1}(\xi)<\delta^{2}(\xi)$. We say that
$\mathcal{D}={\{([t_{i-1},t_{i}]_{\mathbb{T}},\xi_{i})\}}^{n}_{i=1}$ is
\begin{itemize}
\item[$(1)$] a partial partition of $[a,b]_{\mathbb{T}}$
if $\bigcup^{n}_{i=1}[t_{i-1},t_{i}]_{\mathbb{T}}\subset [a,b]_{\mathbb{T}}$;
\item[$(2)$] a partition of $[a,b]_{\mathbb{T}}$
if $\bigcup^{n}_{i=1}[t_{i-1},t_{i}]_{\mathbb{T}}=[a,b]_{\mathbb{T}}$;
\item[$(3)$] a $\delta$-fine Henstock--Kurzweil (HK) partition of
$[a,b]_{\mathbb{T}}$ if
$\xi_{i}\in[t_{i-1},t_{i}]_{\mathbb{T}}\subset(\xi_{i}-\delta_{L}(\xi_{i}),
\xi_{i}+\delta_{R}(\xi_{i}))_{\mathbb{T}}$ for all $ i=1,2,\ldots, n$.
\end{itemize}
Given a $\delta$-fine HK partition $\mathcal{D}={\{([t_{i-1},t_{i}]_{\mathbb{T}},
\xi_{i})\}}^{n}_{i=1}$ of $[a,b]_{\mathbb{T}}$, we write
$$
S(f,\mathcal{D},\delta)=\sum^{n}_{i=1}f(\xi_{i})(t_{i}-t_{i-1})
$$
for integral sums over $D$, whenever $f:[a,b]_{\mathbb{T}}\rightarrow X$.
In what follows, we shall always assume that $X$ is a Dedekind complete
weakly $\sigma$-distributive Riesz space.


\section{The Henstock--Kurzweil delta integral for Riesz-space-valued functions}
\label{sec:3}

Before formulating and giving the proof of our first result,
we need the following definition.

\begin{definition}
\label{def1}
A function $f:[a,b]_{\mathbb{T}}\rightarrow X$ is called Henstock--Kurzweil
delta integrable (HK $\Delta$-integrable) on $[a,b]_{\mathbb{T}}$,
if there exists $x\in X$ and a $(D)$-sequence $(a_{i,j})_{i,j}$
of elements of $X$ such that for every $\varphi \in\mathbb{N}^{\mathbb{N}}$
there exists a $\Delta$-gauge, $\delta$, for $[a,b]_{\mathbb{T}}$, such that
$$
\left|S(f,\mathcal{D},\delta)-x\right|
<\bigvee_{i=1}^{\infty}a_{i,\varphi(i)}
$$
for each $\delta$-fine HK partition
$\mathcal{D}={\{([t_{i-1},t_{i}]_{\mathbb{T}},\xi_{i})\}}^{n}_{i=1}$
of $[a,b]_{\mathbb{T}}$. In this case, $x$ is called the HK $\Delta$-integral
of $f$ on $[a,b]_{\mathbb{T}}$ and is denoted by $x=\int_{a}^{b}f(t)\Delta t$.
\end{definition}

\begin{theorem}
\label{th1}
If $f:[a,b]_{\mathbb{T}}\rightarrow X$ is HK $\Delta$-integrable on $[a,b]_{\mathbb{T}}$,
then the integral of $f$ is determined uniquely.
\end{theorem}

\begin{proof}
Suppose there exist $x_{1},\ x_{2}\in X$ and $(D)$-sequences $(a_{i,j})_{i,j}, (b_{i,j})_{i,j}$
of elements of $X$ and for every $\varphi \in\mathbb{N}^{\mathbb{N}}$ there exist
two $\Delta$-gauges, $\delta_{1}, \delta_{2}$, for $[a,b]_{\mathbb{T}}$, such that
$$
\left|S(f,\mathcal{D}_{1},\delta_{1})-x_{1}\right|
<\bigvee_{i=1}^{\infty}a_{i,\varphi(i)},
\quad
\left|S(f,\mathcal{D}_{2},\delta_{2})-x_{2}\right|
<\bigvee_{i=1}^{\infty}b_{i,\varphi(i)}
$$
for each $\delta_{1}$-fine HK partition
$\mathcal{D}_{1}$ and $\delta_{2}$-fine HK partition
$\mathcal{D}_{2}$, respectively. Let $\delta=\min\{\delta_{1}, \delta_{2}\}$
and $(c_{i,j})_{i,j}$ be a $(D)$-sequence of elements of $X$ such that
$$
\bigvee_{i=1}^{\infty}a_{i,\varphi(i)}+\bigvee_{i=1}^{\infty}
b_{i,\varphi(i)}\leq \bigvee_{i=1}^{\infty}c_{i,\varphi(i)}
$$
for every $\varphi \in\mathbb{N}^{\mathbb{N}}$. Then,
$$
\left|x_{1}-x_{2}|\leq|S(f,\mathcal{D},\delta)-x_{1}\right|
+\left|S(f,\mathcal{D},\delta)-x_{2}\right|
<\bigvee_{i=1}^{\infty}a_{i,\varphi(i)}
+\bigvee_{i=1}^{\infty}b_{i,\varphi(i)}
\leq \bigvee_{i=1}^{\infty}c_{i,\varphi(i)}
$$
for each $\delta$-fine HK partition $\mathcal{D}$.
Since $X$ is weak $\sigma$-distributive, we obtain that
$$
\left|x_{1}-x_{2}\right|
\leq\bigwedge_{\varphi \in\mathbb{N}^{\mathbb{N}}}
\left(\bigvee_{i=1}^{\infty}c_{i,\varphi(i)}\right)=0.
$$
The proof is complete.
\end{proof}

\begin{theorem}
\label{th2}
If $f, g:[a,b]_{\mathbb{T}}\rightarrow X$ are HK $\Delta$-integrable
on $[a,b]_{\mathbb{T}}$ and $\alpha,\beta\in \mathbb{R}$, then
$\alpha f+\beta g$ is HK $\Delta$-integrable on $[a,b]_{\mathbb{T}}$
and
$$
\int_{a}^{b}(\alpha f(t)+\beta g(t))\Delta t
=\alpha\int_{a}^{b}f(t)\Delta t+\beta\int_{a}^{b}g(t)\Delta t.
$$
\end{theorem}

\begin{proof}
We shall prove that if $f, g$ are HK $\Delta$-integrable on
$[a,b]_{\mathbb{T}}$ and $c\in \mathbb{R}$, then $f+g$ and
$cf$ are HK $\Delta$-integrable too and
$$
\int_{a}^{b}(f(t)+g(t))\Delta t=\int_{a}^{b}f(t)\Delta t+\int_{a}^{b}g(t)\Delta t,
\quad
\int_{a}^{b}c f(t)\Delta t=c\int_{a}^{b}f(t)\Delta t.
$$
If $f$ is HK $\Delta$-integrable on $[a,b]_{\mathbb{T}}$, then
there exists a $(D)$-sequence $(a_{i,j})_{i,j}$ of elements of $X$
such that for every $\varphi \in\mathbb{N}^{\mathbb{N}}$ there exists
a $\Delta$-gauge, $\delta_{1}$, for $[a,b]_{\mathbb{T}}$, such that
$$
\left|S(f,\mathcal{D}_{1},\delta_{1})-\int_{a}^{b}f(t)\Delta t\right|
<\bigvee_{i=1}^{\infty}a_{i,\varphi(i)}
$$
for each $\delta_{1}$-fine HK partition $\mathcal{D}_{1}$. Similarly,
there exists a $(D)$-sequence $(b_{i,j})_{i,j}$ of elements of $X$
such that for every $\varphi \in\mathbb{N}^{\mathbb{N}}$ there exists
a $\Delta$-gauge, $\delta_{2}$, for $[a,b]_{\mathbb{T}}$, such that
$$
\left|S(g,\mathcal{D}_{2},\delta_{2})-\int_{a}^{b}g(t)\Delta t\right|
<\bigvee_{i=1}^{\infty}b_{i,\varphi(i)}
$$
for each $\delta_{2}$-fine HK partition $\mathcal{D}_{2}$.
Let $\delta=\min\{\delta_{1}, \delta_{2}\}$,
and consider a $(D)$-sequence $(c_{i,j})_{i,j}$ of elements of $X$ such that
$$
\bigvee_{i=1}^{\infty}a_{i,\varphi(i)}+\bigvee_{i=1}^{\infty}b_{i,\varphi(i)}
\leq \bigvee_{i=1}^{\infty}c_{i,\varphi(i)}
$$
for every $\varphi \in\mathbb{N}^{\mathbb{N}}$. Then,
\begin{eqnarray*}
\left|S(f+g,\mathcal{D},\delta)-\int_{a}^{b}f(t)\Delta t-\int_{a}^{b}g(t)\Delta t\right|
&=&\left|S(f,\mathcal{D},\delta)-\int_{a}^{b}f(t)\Delta t+S(g,\mathcal{D},\delta)
-\int_{a}^{b}g(t)\Delta t\right| \\
&\leq&\left|S(f,\mathcal{D},\delta)-\int_{a}^{b}f(t)\Delta t\right|
+\left|S(g,\mathcal{D},\delta)-\int_{a}^{b}g(t)\Delta t\right|\\
&<&\bigvee_{i=1}^{\infty}a_{i,\varphi(i)}+\bigvee_{i=1}^{\infty}b_{i,\varphi(i)}\\
&\leq& \bigvee_{i=1}^{\infty}c_{i,\varphi(i)}
\end{eqnarray*}
for each $\delta$-fine HK partition $\mathcal{D}$. Hence, $f+g$ is HK $\Delta$-integrable and
$$
\int_{a}^{b}(f(t)+g(t))\Delta t=\int_{a}^{b}f(t)\Delta t+\int_{a}^{b}g(t)\Delta t.
$$
For $c\in \mathbb{R}$, $(|c|a_{i,j})_{i,j}$ is a $(D)$-sequence. Then,
$$
\left|S(cf,\mathcal{D},\delta)-c\int_{a}^{b}f(t)\Delta t\right|
\leq |c|\left|S(f,\mathcal{D},\delta)-\int_{a}^{b}f(t)\Delta t\right|
<|c|\bigvee_{i=1}^{\infty}a_{i,\varphi(i)}
=\bigvee_{i=1}^{\infty}|c|a_{i,\varphi(i)}
$$
for each $\delta$-fine HK partition $\mathcal{D}$. This implies
that $cf$ is HK $\Delta$-integrable and
$$
\int_{a}^{b}c f(t)\Delta t=c\int_{a}^{b}f(t)\Delta t.
$$
The proof is complete.
\end{proof}

\begin{theorem}[Cauchy--Bolzano condition]
\label{th3}
A function $f:[a,b]_{\mathbb{T}}\rightarrow X$ is HK $\Delta$-integrable
on $[a,b]_{\mathbb{T}}$ if and only if there exists a $(D)$-sequence
$(a_{i,j})_{i,j}$ of elements of $X$ such that for every
$\varphi \in\mathbb{N}^{\mathbb{N}}$ there exists a $\Delta$-gauge, $\delta$,
for $[a,b]_{\mathbb{T}}$, such that
$$
\left|S(f,\mathcal{D}_{1},\delta)-S(f,\mathcal{D}_{2},\delta)\right|
<\bigvee_{i=1}^{\infty}a_{i,\varphi(i)}
$$
for each $\delta$-fine HK partition $\mathcal{D}_{1},\mathcal{D}_{2}$
of $[a,b]_{\mathbb{T}}$.
\end{theorem}

\begin{proof}
(Necessity). This follows from the inequality
$$
\left|S(f,\mathcal{D}_{1},\delta)-S(f,\mathcal{D}_{2},\delta)\right|
\leq \left|S(f,\mathcal{D}_{1},\delta)-\int_{a}^{b}f(t)\Delta t\right|
+\left|S(f,\mathcal{D}_{2},\delta)-\int_{a}^{b}f(t)\Delta t\right|
$$
and some routine arguments. (Sufficiency). To every $\varphi \in\mathbb{N}^{\mathbb{N}}$,
there exists a $\Delta$-gauge $\delta_{\varphi}(\xi)$ with the following property. Let
$$
\delta_{[a,b]_{\mathbb{T}}}=\left\{\delta(\xi):\exists \varphi
\in\mathbb{N}^{\mathbb{N}}, \delta(\xi)
=\delta_{\varphi}(\xi), \xi\in[a,b]_{\mathbb{T}}\right\}.
$$
Then, for $\delta(\xi)\in\delta_{[a,b]_{\mathbb{T}}}$ and a $\delta$-fine HK partition $\mathcal{D}$,
the set $\left\{S(f,\mathcal{D},\delta)\right\}$ is bounded. Indeed, for $X$ boundedly complete,
there exist
$$
a_{\delta}=\bigwedge_{\mathcal{D}}S(f,\mathcal{D},\delta),
\quad
b_{\delta}=\bigvee_{\mathcal{D}}S(f,\mathcal{D},\delta).
$$
For $\delta_{1}(\xi),\delta_{2}(\xi)\in\delta_{[a,b]_{\mathbb{T}}}$,
let $\delta(\xi)=\min\{\delta_{1}(\xi),\delta_{2}(\xi)\}$. Then,
$$
a_{\delta_{1}}=\bigwedge_{\mathcal{D}}S(f,\mathcal{D},\delta_{1})
\leq\bigwedge_{\mathcal{D}}S(f,\mathcal{D},\delta)
\leq\bigvee_{\mathcal{D}}S(f,\mathcal{D},\delta)
\leq \bigvee_{\mathcal{D}}S(f,\mathcal{D},\delta_{2})=b_{\delta_{2}}.
$$
Therefore,
$$
\bigvee_{\delta(\xi)\in\delta_{[a,b]_{\mathbb{T}}}}a_{\delta}\leq\bigwedge_{\delta(\xi)
\in\delta_{[a,b]_{\mathbb{T}}}}b_{\delta}.
$$
Hence, there exists $x\in X$ such that $a_{\delta}\leq x\leq b_{\delta}$
for all $\delta(\xi)\in\delta_{[a,b]_{\mathbb{T}}}$. Now, let
$\varphi \in\mathbb{N}^{\mathbb{N}}$. Then there exists a $\Delta$-gauge
$\delta_{\varphi}(\xi)$ for $[a,b]_{\mathbb{T}}$ such that
$$
S(f,\mathcal{D}_{1},\delta_{\varphi})
\leq S(f,\mathcal{D}_{2},\delta_{\varphi})
+\bigvee_{i=1}^{\infty}a_{i,\varphi(i)}
$$
for each $\delta_{\varphi}$-fine HK partition $\mathcal{D}_{1},\mathcal{D}_{2}$.
Fix $\mathcal{D}_{2}$. Then
$$
b_{\delta_{\varphi}}\leq S(f,\mathcal{D}_{2},\delta_{\varphi})
+\bigvee_{i=1}^{\infty}a_{i,\varphi(i)}.
$$
Since the inequality holds for every $\delta_{\varphi}$-fine
HK partition $\mathcal{D}_{2}$, we have
$$
b_{\delta_{\varphi}}\leq a_{\delta_{\varphi}}
+\bigvee_{i=1}^{\infty}a_{i,\varphi(i)}.
$$
By the weak $\sigma$-distributivity of $X$, we obtain that
$$
\bigwedge_{\varphi \in\mathbb{N}^{\mathbb{N}}}\left(
\bigvee_{i=1}^{\infty}a_{i,\varphi(i)}\right)=0
$$
and so
$$
\bigwedge_{\varphi \in\mathbb{N}^{\mathbb{N}}}b_{\delta_{\varphi}}
-\bigvee_{\varphi \in\mathbb{N}^{\mathbb{N}}}a_{\delta_{\varphi}}
\leq\bigwedge_{\varphi \in\mathbb{N}^{\mathbb{N}}}(b_{\delta_{\varphi}}
-a_{\delta_{\varphi}})=0.
$$
Consequently,
$$
x=\bigwedge_{\varphi \in\mathbb{N}^{\mathbb{N}}}b_{\delta_{\varphi}}
=\bigvee_{\varphi \in\mathbb{N}^{\mathbb{N}}}a_{\delta_{\varphi}}.
$$
Then, for every $\delta_{\varphi}$-fine HK partition $\mathcal{D}$, we have
$$
S(f,\mathcal{D},\delta_{\varphi})-x\leq b_{\delta_{\varphi}}
-a_{\delta_{\varphi}}\leq \bigvee_{i=1}^{\infty}a_{i,\varphi(i)},
$$
$$
x- S(f,\mathcal{D},\delta_{\varphi})
\leq b_{\delta_{\varphi}}-a_{\delta_{\varphi}}
\leq \bigvee_{i=1}^{\infty}a_{i,\varphi(i)}.
$$
It follows that
$$
\left|S(f,\mathcal{D},\delta_{\varphi})-x\right|
\leq \bigvee_{i=1}^{\infty}a_{i,\varphi(i)}
$$
and the proof is complete.
\end{proof}

\begin{theorem}
\label{th4}
If $f:[a,b]_{\mathbb{T}}\rightarrow X$, then $f$
is HK-$\Delta$ integrable on $[a,b]_{\mathbb{T}}$
if and only if $f$ is HK-$\Delta$ integrable on
$[a,c]_{\mathbb{T}}$ and $[c,b]_{\mathbb{T}}$.
Moreover, in this case
$$
\int_{a}^{b}f(t)\Delta t=\int_{a}^{c}f(t)\Delta t+\int_{c}^{b}f(t)\Delta t.
$$
\end{theorem}

\begin{proof}
(Necessity). By Theorem~\ref{th3}, there exists a $(D)$-sequence $(a_{i,j})_{i,j}$
of elements of $X$ such that for every $\varphi \in\mathbb{N}^{\mathbb{N}}$
there exists a $\Delta$-gauge, $\delta$, for $[a,b]_{\mathbb{T}}$, such that
$$
\left|S(f,\mathcal{D}_{1},\delta)-S(f,\mathcal{D}_{2},\delta)\right|
<\bigvee_{i=1}^{\infty}a_{i,\varphi(i)}$$ for each
$\delta$-fine HK partition $\mathcal{D}_{1},\mathcal{D}_{2}$
of $[a,b]_{\mathbb{T}}$. Take any two $\delta$-fine HK partition of
$[a,c]_{\mathbb{T}}$, say, $\mathcal{D}_{3}$ and $\mathcal{D}_{4}$.
Similarly, take another $\delta$-fine HK partition $\mathcal{D}_{5}$ of $[c,b]_{\mathbb{T}}$.
Then, we have
$$
\left|S(f,\mathcal{D}_{3},\delta)-S(f,\mathcal{D}_{4},\delta)\right|
=\left|S(f,\mathcal{D}_{3}+\mathcal{D}_{5},\delta)
-S(f,\mathcal{D}_{5},\delta)+S(f,\mathcal{D}_{5},\delta)
-S(f,\mathcal{D}_{4}+\mathcal{D}_{5},\delta)\right|
<\bigvee_{i=1}^{\infty}a_{i,\varphi(i)}.
$$
Hence, $f$ is HK-$\Delta$ integrable on $[a,c]_{\mathbb{T}}$.
Similarly, $f$ is HK-$\Delta$ integrable on $[c,b]_{\mathbb{T}}$.
Consequently, there exist $(D)$-sequences $(a_{i,j})_{i,j}$, $(b_{i,j})_{i,j}$
and $(c_{i,j})_{i,j}$ of elements of $X$ such that for every
$\varphi \in\mathbb{N}^{\mathbb{N}}$ there exists a $\Delta$-gauge,
$\delta$, for $[a,b]_{\mathbb{T}}$, such that
$$
\left|S(f,\mathcal{D},\delta)-\int_{a}^{b}f(t)\Delta t\right|
<\bigvee_{i=1}^{\infty}a_{i,\varphi(i)},
$$
$$
\left|S(f,\mathcal{D'},\delta)-\int_{a}^{c}f(t)\Delta t\right|
<\bigvee_{i=1}^{\infty}b_{i,\varphi(i)},
$$
$$
\left|S(f,\mathcal{D}-\mathcal{D}',\delta)-\int_{c}^{b}f(t)\Delta t\right|
<\bigvee_{i=1}^{\infty}c_{i,\varphi(i)},
$$
for each $\delta$-fine HK partition $\mathcal{D}$ of $[a,b]_{\mathbb{T}}$,
$\mathcal{D'}$, of $[a,c]_{\mathbb{T}}$ and $\mathcal{D}-\mathcal{D}'$
of $[c,b]_{\mathbb{T}}$. Then, there exists a $(D)$-sequence $(d_{i,j})_{i,j}$
of elements of $X$ such that
\begin{equation*}
\begin{split}
\Bigg|\int_{a}^{b}f(t)\Delta &t
-\int_{a}^{c}f(t)\Delta t-\int_{c}^{b}f(t)\Delta t\Bigg|\\
&\leq\left|S(f,\mathcal{D},\delta)-\int_{a}^{b}f(t)\Delta t\right|
+\left|S(f,\mathcal{D'},\delta)-\int_{a}^{c}f(t)\Delta t\right|
+ \left|S(f,\mathcal{D}-\mathcal{D}',\delta)-\int_{c}^{b}f(t)\Delta t\right|\\
&\leq\bigvee_{i=1}^{\infty}a_{i,\varphi(i)}
+\bigvee_{i=1}^{\infty}b_{i,\varphi(i)}
+\bigvee_{i=1}^{\infty}c_{i,\varphi(i)}\leq\bigvee_{i=1}^{\infty}d_{i,\varphi(i)}
\end{split}
\end{equation*}
and the result follows. (Sufficiency). Let $f$ be HK-$\Delta$ integrable on
$[a,c]_{\mathbb{T}}$ and $[c,b]_{\mathbb{T}}$. Then there exist
$(D)$-sequences $(a_{i,j})_{i,j},(b_{i,j})_{i,j}$ of elements of $X$
such that for every $\varphi \in\mathbb{N}^{\mathbb{N}}$ there exist $\Delta$-gauge,
$$
\delta^{1}(\xi)=\left(\delta^{1}_{L}(\xi),\delta^{1}_{R}(\xi)\right),
\quad \delta^{2}(\xi)=\left(\delta^{2}_{L}(\xi),\delta^{2}_{R}(\xi)\right),
$$
for $[a,b]_{\mathbb{T}}$ such that
$$
\left|S(f,\mathcal{D}_{1},\delta^{1})
-\int_{a}^{c}f(t)\Delta t\right|
<\bigvee_{i=1}^{\infty}a_{i,\varphi(i)},
\quad
\left|S(f,\mathcal{D}_{2},\delta^{2})
-\int_{c}^{b}f(t)\Delta t\right|
<\bigvee_{i=1}^{\infty}b_{i,\varphi(i)}
$$
for each $\delta^{1}$-fine HK partition
$\mathcal{D}_{1}={\{([t^{1}_{k-1},t^{1}_{k}]_{\mathbb{T}},\xi^{1}_{k})\}}^{n}_{k=1}$
of $[a,c]_{\mathbb{T}}$ and for each $\delta^{2}$-fine HK partition
$\mathcal{D}_{2}={\{([t^{2}_{k-1},t^{2}_{k}]_{\mathbb{T}},\xi^{2}_{k})\}}^{m}_{k=1}$
of $[c,b]_{\mathbb{T}}$, respectively. We define a $\Delta$-gauge,
$\delta(\xi)=(\delta_{L}(\xi),\delta_{R}(\xi))$, on $[a,b]_{\mathbb{T}}$,
by first defining $\delta_{L}(\xi)$ as
\begin{equation*}
\delta_{L}(\xi)=
\begin{cases}
\delta^{1}_{L}(\xi),& \text{if $\xi\in [a,c)_{\mathbb{T}}$},\\
\delta^{1}_{L}(\xi),& \text{if $\xi=c=\rho(c)$},\\
\min\left\{\delta^{1}_{L}(\xi),\frac{\eta(c)}{2}\right\},& \text{if $\xi=c>\rho(c)$},\\
\min\left\{\delta^{2}_{L}(\xi),\frac{\xi-c}{2}\right\},& \text{if $\xi\in (c,b]_{\mathbb{T}}$},
\end{cases}
\end{equation*}
and then defining $\delta_{R}(\xi)$ as
\begin{equation*}
\delta_{R}(\xi)=
\begin{cases}
\min\left\{\delta^{1}_{R}(\xi),\max\left\{\mu(\xi),\frac{c-\xi}{2}\right\}\right\},
& \text{if $\xi\in [a,c)_{\mathbb{T}}$},\\
\min\{\delta^{2}_{R}(\xi)\},& \text{if $\xi\in [c,b]_{\mathbb{T}}$}.
\end{cases}
\end{equation*}
Now, let $\mathcal{D}={\{([t_{k-1},t_{k}]_{\mathbb{T}},\xi_{k})\}}^{p}_{k=1}$
be a $\delta$-fine HK partition of $[a,b]_{\mathbb{T}}$. Then,
either $c$ is a tag point for $\mathcal{D}$, say $c=\xi_{q}$,
and $t_{q}>c$; or $\rho(c)<c$, and $\rho(c)$ is a tag point for $\mathcal{D}$,
say $\rho(c)=\xi_{q}$, and $t_{q}=c$. In the first case, there exists a
$(D)$-sequence $(c_{i,j})_{i,j}$ of elements of $X$ such that for every
$\varphi \in\mathbb{N}^{\mathbb{N}}$ we have
\begin{eqnarray*}
\Bigg|S(f,\mathcal{D},\delta)
&-&\int_{a}^{c}f(t)\Delta t-\int_{c}^{b}f(t)\Delta t\Bigg|\\
&=&\left|\sum_{k=1}^{p}f(\xi_{k})(t_{k}-t_{k-1})
-\int_{a}^{c}f(t)\Delta t-\int_{c}^{b}f(t)\Delta t\right|\\
&\leq&\left|\sum_{k=1}^{q-1}f(\xi_{k})(t_{k}-t_{k-1})+f(c)(c-t_{q-1})
-\int_{a}^{c}f(t)\Delta t\right|\\
&&+\left|\sum_{k=q+1}^{p}f(\xi_{k})(t_{k}-t_{k-1})+f(c)(t_{q}-c)
-\int_{c}^{b}f(t)\Delta t\right|\\
&<& \bigvee_{i=1}^{\infty}a_{i,\varphi(i)}+\bigvee_{i=1}^{\infty}b_{i,\varphi(i)}
<\bigvee_{i=1}^{\infty}c_{i,\varphi(i)}.
\end{eqnarray*}
Using the weak $\sigma$-distributivity, we get the corresponding results.
The other case is easy and is omitted. Hence, $f$ is HK-$\Delta$ integrable
on $[a,b]_{\mathbb{T}}$ and
$\displaystyle \int_{a}^{b}f(t)\Delta t=\int_{a}^{c}f(t)\Delta t+\int_{c}^{b}f(t)\Delta t$.
This concludes the proof.
\end{proof}

\begin{lemma}[The Saks--Henstock lemma]
\label{lem1}
Let $f:[a,b]_{\mathbb{T}}\rightarrow X$ be HK-$\Delta$ integrable
on $[a,b]_{\mathbb{T}}$. Then there exists a $(D)$-sequence
$(a_{i,j})_{i,j}$ of elements of $X$ such that
for every $\varphi \in\mathbb{N}^{\mathbb{N}}$ there exists
a $\Delta$-gauge, $\delta$, for $[a,b]_{\mathbb{T}}$, such that
$$
\left|S(f,\mathcal{D},\delta)-\int_{a}^{b}f(t)\Delta t\right|
<\bigvee_{i=1}^{\infty}a_{i,\varphi(i)}
$$
for each $\delta$-fine HK partition
$\mathcal{D}$ of $[a,b]_{\mathbb{T}}$.
In particular, if $\mathcal{D}^{'}={\{([t_{k-1},t_{k}]_{\mathbb{T}},\xi_{k})\}}^{m}_{k=1}$
is an arbitrary $\delta$-fine partial HK partition of $[a,b]_{\mathbb{T}}$, then
$$
\left|S(f,\mathcal{D}^{'},\delta)
-\sum_{k=1}^{m}\int_{t_{k-1}}^{t_{k}}f(t)\Delta t\right|
\leq\bigvee_{i=1}^{\infty}a_{i,\varphi(i)}.
$$
\end{lemma}

\begin{proof}
Assume $\mathcal{D}^{'}={\{([t_{k-1},t_{k}]_{\mathbb{T}},\xi_{k})\}}^{m}_{k=1}$
is an arbitrary $\delta'$-fine partial HK partition of $[a,b]_{\mathbb{T}}$. Then the
complement $[a,b]_{\mathbb{T}}\backslash\bigcup^{m}_{k=1} [t_{k-1},k_{i}]_{\mathbb{T}}$
can be expressed as a fine collection of closed subintervals and we denote
$$
[a,b]_{\mathbb{T}}\backslash\bigcup^{m}_{k=1} [t_{k-1},t_{k}]_{\mathbb{T}}
=\sum_{k=1}^{n}[t^{'}_{k-1},t^{'}_{k}]_{\mathbb{T}}.
$$
From Theorem~\ref{th4}, we know that $\int_{t^{'}_{k-1}}^{t^{'}_{k}}f(t)\Delta t$ exists.
Then, there exist $(D)$-sequences $(b_{k,i,j})_{k,i,j}$ of elements of $X$ such that for every
$\varphi \in\mathbb{N}^{\mathbb{N}}$ there exists $\Delta$-gauges,
$\delta_{1},\delta_{2},\ldots,\delta_{n}$, for $[a,b]_{\mathbb{T}}$, such that
$$
\left|S(f,\mathcal{D}_{k},\delta_{k})
-\int_{t^{'}_{k-1}}^{t^{'}_{k}}f(t)\Delta t\right|
<\bigvee_{i=1}^{\infty}b_{k,i,\varphi(i)}
$$
for each $\delta_{k}$-fine HK partition $\mathcal{D}_{k}$
of $[t_{k-1},t_{k}]_{\mathbb{T}}$. Assume that  $\delta\leq \delta',
\delta_{1},\delta_{2},\ldots,\delta_{n}$. Let
$$
\mathcal{D}_{0}=\mathcal{D}^{'}+\mathcal{D}_{1}+\mathcal{D}_{2}
+\cdots+\mathcal{D}_{n}.
$$
Obviously, $\mathcal{D}_{0}$ is a $\delta$-fine
HK partition of $[a,b]_{\mathbb{T}}$. Then, there exists a $(D)$-sequence
$(a_{i,j})_{i,j}$ of elements of $X$ such that
\begin{equation*}
\begin{split}
\left|S(f,\mathcal{D}_{0},\delta)-\int_{a}^{b}f(t)\Delta t\right|
&=\left|S(f,\mathcal{D}^{'},\delta)+\sum^{n}_{k=1}
S(f,\mathcal{D}_{k},\delta)-\int_{a}^{b}f(t)\Delta t\right|<\bigvee_{i=1}^{\infty}a_{i,\varphi(i)}
\end{split}
\end{equation*}
for every $\varphi \in\mathbb{N}^{\mathbb{N}}$.
Consequently, we obtain
\begin{eqnarray*}
\Bigg|S(f,\mathcal{D}^{'},\delta)-\sum_{k=1}^{m}
\int_{t_{k-1}}^{t_{k}}f(t)\Delta t\Bigg|&=&\left|S(f,\mathcal{D}_{0},\delta)-\sum^{n}_{k=1}
S(f,\mathcal{D}_{k},\delta)-\left(\int_{a}^{b}f(t)\Delta t-\sum^{n}_{k=1}
\int_{t^{'}_{k-1}}^{t^{'}_{k}}f(t)\Delta t\right)\right|\\
&\leq&\left|S(f,\mathcal{D}_{0},\delta)-\int_{a}^{b}f(t)\Delta t\right|
+ \sum^{n}_{k=1} \left|S(f,\mathcal{D}_{k},\delta)
-\int_{t^{'}_{k-1}}^{t^{'}_{k}}f(t)\Delta t\right|\\
&< &\bigvee_{i=1}^{\infty}a_{i,\varphi(i)}
+ \sum^{n}_{k=1}\bigvee_{i=1}^{\infty}b_{k,i,\varphi(i)}\\
&\leq &\bigvee_{i=1}^{\infty}a_{i,\varphi(i)}
+ \sum^{n}_{k=1}\bigvee_{i=1}^{\infty}\sum^{n}_{k=1}b_{k,i,\varphi(i)}\\
&\leq&\bigvee_{i=1}^{\infty}a_{i,\varphi(i)}
+ n\bigvee_{i=1}^{\infty}\sum^{n}_{k=1}b_{k,i,\varphi(i)}.
\end{eqnarray*}
Let $c_{i,j}=n\sum^{n}_{k=1}b_{k,i,\varphi(j)}$. Then,
$(c_{i,j})_{i,j}$ is a $(D)$-sequence and, for every
$\varphi \in\mathbb{N}^{\mathbb{N}}$, we have
$$
\left|S(f,\mathcal{D}^{'},\delta)
-\sum_{k=1}^{m}\int_{t_{k-1}}^{t_{k}}f(t)\Delta t\right|
\leq\bigvee_{i=1}^{\infty}a_{i,\varphi(i)}.
$$
The proof is complete.
\end{proof}


\section{Convergence theorems}
\label{sec:4}

In this section we prove two convergence theorems.
We begin with the following two definitions.

\begin{definition}
\label{def2}
We say that $f_{n}\rightarrow f$ converges with a common regulating
sequence (w.c.r.s.) if there exists a $(D)$-sequence $(a_{i,j})_{i,j}$
of elements of $X$ such that for every $\varphi \in\mathbb{N}^{\mathbb{N}}$
and every $t\in [a,b]_{\mathbb{T}}$ there exists $p=p(t)$ such that
$$
|f_{n}(t)-f(t)|<\bigvee_{i=1}^{\infty}a_{i,\varphi(i)}
$$
for any $n\geq p$.
\end{definition}

\begin{definition}
\label{def3}
We say that $\{f_{n}\}_{n=1}^{\infty}$ is uniformly HK $\Delta$-integrable
on $[a,b]_{\mathbb{T}}$ if each $f_{n}$ is HK $\Delta$-integrable on
$[a,b]_{\mathbb{T}}$ and there exists a $(D)$-sequence $(a_{i,j})_{i,j}$
of elements of $X$ such that for every $\varphi \in\mathbb{N}^{\mathbb{N}}$
there exists a $\Delta$-gauge, $\delta$, for $[a,b]_{\mathbb{T}}$, such that
$$
\left|S(f_{n},\mathcal{D},\delta)
-\int_{a}^{b}f_{n}(t)\Delta t\right|
<\bigvee_{i=1}^{\infty}b_{i,\varphi(i)}
$$
for each $\delta$-fine HK partition $\mathcal{D}$ of
$[a,b]_{\mathbb{T}}$ and $n\in \mathbb{N}$.
\end{definition}

For uniformly HK $\Delta$-integrable sequences of integrable functions,
we have the following convergence theorem.

\begin{theorem}
\label{th5}
Let $\{f_{n}\}_{n=1}^{\infty}$ be a sequence of uniformly HK $\Delta$-integrable
functions on $[a,b]_{\mathbb{T}}$ and assume that $f_{n}\rightarrow f$ converges
with a common regulating sequence. Then $f$ is HK $\Delta$-integrable and
$$
\lim_{n\rightarrow\infty}\int_{a}^{b}f_{n}(t)\Delta t
=\int_{a}^{b}f(t)\Delta t.
$$
\end{theorem}

\begin{proof}
We will prove the Theorem in two steps. Step 1.
By assumption, there exists a $(D)$-sequence $(b_{i,j})_{i,j}$
of elements of $X$ such that for every $\varphi \in\mathbb{N}^{\mathbb{N}}$
there exist two $\Delta$-gauges, $\delta_{1}$ and $\delta_{2}$, for $[a,b]_{\mathbb{T}}$,
such that
$$
\left|S(f_{n},\mathcal{D}_{1},\delta_{1})
-\int_{a}^{b}f_{n}(t)\Delta t\right|
<\bigvee_{i=1}^{\infty}b_{i,\varphi(i+1)},
\quad
\left|S(f_{n},\mathcal{D}_{2},\delta_{2})
-\int_{a}^{b}f_{n}(t)\Delta t\right|
<\bigvee_{i=1}^{\infty}b_{i,\varphi(i+2)}
$$
for each $\delta_{1}$-fine ($\delta_{2}$-fine)
HK partition $\mathcal{D}_{1}=\{[t_{i-1},t_{i}],\xi_{i}\}_{i}$
($\mathcal{D}_{2}=\{[t'_{i-1},t'_{i}],\xi'_{i}\}_{i}$)
of $[a,b]_{\mathbb{T}}$ and $n\in \mathbb{N}$.
Let $\delta=\min\left\{\delta_{1},\delta_{2}\right\}$.
By the w.c.r.s. convergence,
$$
\left|S(f,\mathcal{D}_{1},\delta)-S(f_{n},\mathcal{D}_{1},\delta)\right|
\leq\sum_{i}|f(\xi_{i})-f_{n}(\xi_{i})|(t_{i}-t_{i-1})
<(b-a)\bigvee_{i=1}^{\infty}a_{i,\varphi(i+1)}
$$
for each $\delta$-fine HK partition $\mathcal{D}_{1}$ and $n\geq p_{1}$.
Similarly, we have
$$
\left|S(f,\mathcal{D}_{2},\delta)-S(f_{n},\mathcal{D}_{2},\delta)\right|
\leq\sum_{i}|f(\xi'_{i})-f_{n}(\xi'_{i})|(t'_{i}-t'_{i-1})
<(b-a)\bigvee_{i=1}^{\infty}a_{i,\varphi(i+2)}
$$
for each $\delta$-fine HK partition $\mathcal{D}_{2}$ and $n\geq p_{2}$.
We can choose a $(D)$-sequence $(c_{i,j})_{i,j}$ of elements of $X$ such that
$$
\bigvee_{i=1}^{\infty}b_{i,\varphi(i+1)}
+\bigvee_{i=1}^{\infty}b_{i,\varphi(i+2)}
+(b-a)\bigvee_{i=1}^{\infty}a_{i,\varphi(i+1)}
+(b-a)\bigvee_{i=1}^{\infty}a_{i,\varphi(i+2)}
\leq\bigvee_{i=1}^{\infty}c_{i,\varphi(i)}.
$$
Let $n>\max\{p_{1},p_{2}\}$. Then,
\begin{eqnarray*}
\left|S(f,\mathcal{D}_{1},\delta)-S(f,\mathcal{D}_{2},\delta)\right|
&=&\left|S(f,\mathcal{D}_{1},\delta)-S(f_{n},\mathcal{D}_{1},\delta)\right|
+\left|S(f_{n},\mathcal{D}_{1},\delta)-\int_{a}^{b}f_{n}(t)\Delta t\right|\\
&&+\left|\int_{a}^{b}f_{n}(t)\Delta t
-S(f_{n},\mathcal{D}_{2},\delta)\right|
+\left|S(f_{n},\mathcal{D}_{2},\delta)
-S(f,\mathcal{D}_{2},\delta)\right|\\
&<& \bigvee_{i=1}^{\infty}c_{i,\varphi(i)}
\end{eqnarray*}
for each $\delta$-fine HK partition $\mathcal{D}_{1}$
and $\mathcal{D}_{2}$ of $[a,b]_{\mathbb{T}}$. Therefore,
by Theorem~\ref{th3}, $f$ is HK $\Delta$-integrable.\\
Step 2. Since $f$ is HK $\Delta$-integrable, there exists
a $(D)$-sequence $(c_{i,j})_{i,j}$ of elements of $X$ such that for every
$\varphi \in\mathbb{N}^{\mathbb{N}}$ there exists a $\Delta$-gauge,
$\delta'$, for $[a,b]_{\mathbb{T}}$, such that
$$
\left|S(f,\mathcal{D}',\delta')
-\int_{a}^{b}f(t)\Delta t\right|
<\bigvee_{i=1}^{\infty}c_{i,\varphi(i+1)}
$$
for each $\delta'$-fine HK partition $\mathcal{D}'$ of
$[a,b]_{\mathbb{T}}$. By the uniform HK $\Delta$-integrability,
there exists a $(D)$-sequence $(e_{i,j})_{i,j}$ of elements
of $X$ for every $\varphi \in\mathbb{N}^{\mathbb{N}}$ such that
$$
\left|S(f_{n},\mathcal{D}',\delta')
-\int_{a}^{b}f_{n}(t)\Delta t\right|
<\bigvee_{i=1}^{\infty}b_{i,\varphi(i+3)}
$$
for each $\delta'$-fine HK partition $\mathcal{D}'$
of $[a,b]_{\mathbb{T}}$ and $n\in \mathbb{N}$.
By the w.c.r.s. convergence,
$$
\left|S(f,\mathcal{D}',\delta')-S(f_{n},\mathcal{D}',\delta')\right|
<(b-a)\bigvee_{i=1}^{\infty}a_{i,\varphi(i+3)}
$$
for each $\delta'$-fine HK partition $\mathcal{D}'$ and $n\geq p$.
Choose a $(D)$-sequence $(d_{i,j})_{i,j}$ of elements of $X$ such that
$$
\bigvee_{i=1}^{\infty}c_{i,\varphi(i+1)}
+\bigvee_{i=1}^{\infty}b_{i,\varphi(i+3)}
+(b-a)\bigvee_{i=1}^{\infty}a_{i,\varphi(i+3)}
\leq\bigvee_{i=1}^{\infty}d_{i,\varphi(i)}.
$$
Then,
\begin{eqnarray*}
\left|\int_{a}^{b}f(t)\Delta t-\int_{a}^{b}f_{n}(t)\Delta t \right|
&=&\left|\int_{a}^{b}f(t)\Delta t-S(f,\mathcal{D}',\delta')\right|
+\left|S(f,\mathcal{D}',\delta')-S(f_{n},\mathcal{D}',\delta')\right|\\
&&+\left|S(f_{n},\mathcal{D}',\delta')-\int_{a}^{b}f_{n}(t)\Delta t\right|\\
&<& \bigvee_{i=1}^{\infty}d_{i,\varphi(i)}
\end{eqnarray*}
for each $\delta'$-fine HK partition $\mathcal{D}'$ of $[a,b]_{\mathbb{T}}$.
It follows that
$
\lim_{n\rightarrow\infty}\int_{a}^{b}f_{n}(t)\Delta t=\int_{a}^{b}f(t)\Delta t
$
and the proof is complete.
\end{proof}

We now recall the well-known Fremlin lemma.

\begin{lemma}[See \cite{F}]
\label{lem2}
Let $\{(a^{n}_{i,j})_{i,j}:n\in \mathbb{N}\}$ be any countable family of regulators.
Then, for each fixed element $x\in X$, $x\geq0$, there exists a $(D)$-sequence
$(a_{i,j})_{i,j}$ of elements of $X$ such that
$$
x\bigwedge \sum_{i=1}^{\infty}\left(\bigvee_{i=1}^{\infty}
a^{n}_{i,\varphi(i)+n}\right)
\leq \bigvee_{i=1}^{\infty}a_{i,\varphi(i)}
$$
for every $\varphi \in\mathbb{N}^{\mathbb{N}}$.
\end{lemma}

\begin{theorem}[Monotone Convergence Theorem]
\label{th6}
Let $\{f_{n}\}_{n=1}^{\infty}$ be a sequence of HK $\Delta$-integrable
functions on $[a,b]_{\mathbb{T}}$, and $f_{1}$ be bounded from below.
Let $f:[a,b]_{\mathbb{T}}\rightarrow X$ be a bounded function such that
$f_{n}\leq f_{n+1}$ and $f_{n}\rightarrow f$ converges with a common
regulating sequence. Then $f$ is HK $\Delta$-integrable and
$$
\lim_{n\rightarrow\infty}\int_{a}^{b}f_{n}(t)\Delta t=\int_{a}^{b}f(t)\Delta t.
$$
\end{theorem}

\begin{proof}
Since $f_{n}$ are HK $\Delta$-integrable functions on $[a,b]_{\mathbb{T}}$,
there exists a $(D)$-sequence $(a_{n,i,j})_{i,j}$ of elements of $X$
such that for every $\varphi \in\mathbb{N}^{\mathbb{N}}$ there exists
a $\Delta$-gauge, $\delta_{n}$, for $[a,b]_{\mathbb{T}}$, such that
$$
\left|S(f_{n},\mathcal{D}_{n},\delta_{n})
-\int_{a}^{b}f_{n}(t)\Delta t\right|
<\bigvee_{i=1}^{\infty}a_{n,i,\varphi(i+n+1)}
$$
for each $\delta_{n}$-fine HK partition
$\mathcal{D}_{n}$ of $[a,b]_{\mathbb{T}}$.
By the w.c.r.s. convergence, there exists a $(D)$-sequence
$(a_{i,j})_{i,j}$ of elements of $X$ such that for every
$\varphi \in\mathbb{N}^{\mathbb{N}}$ and every
$t\in [a,b]_{\mathbb{T}}$ there exists $p=p(t)$ such that
$$
|f_{n}(t)-f(t)|<\bigvee_{i=1}^{\infty}a_{i,\varphi(i+1)}
$$
for any $n\geq p$. Let
$b_{1,i,j}=2(b-a)a_{i,j}$,
$b_{m,i,j}=a_{m-1,i,j}$, $m=2,3,\ldots$,
and $x=(b-a)(L-l)$, where $L,l\in X$ are such that
$l\leq f_{1}(\xi)\leq f(\xi)\leq L$ for any
$\xi\in [a,b]_{\mathbb{T}}$. By Fremlin's Lemma~\ref{lem2},
there exists a $(D)$-sequence $(b_{i,j})_{i,j}$
of elements of $X$ such that for every
$\varphi \in\mathbb{N}^{\mathbb{N}}$
$$
x\bigwedge\left(\sum_{m=1}^{\infty}\bigvee_{i=1}^{\infty}
b_{m,i,\varphi(i+m)}\right)
\leq\bigvee_{i=1}^{\infty}b_{i,\varphi(i)}.
$$
Let $\varphi \in\mathbb{N}^{\mathbb{N}}$ and
\begin{gather*}
\delta(\xi)=\min\{\delta_{1}(\xi),\delta_{2}(\xi),\ldots,\delta_{p(t)}(\xi)\},\\ \mathcal{D}^{0}=\mathcal{D}^{1}\bigcup\mathcal{D}^{2}, 
\end{gather*}
where
\begin{equation*}
\begin{split}
\mathcal{D}^{1}&=\{([t_{k-1},t_{k}]_{\mathbb{T}},\xi_{k})
\in\mathcal{D}\mid p(\xi_{k})\geq n\},\\
\mathcal{D}^{2}&=\bigcup_{p(\xi_{k})< n}\mathcal{D}_{k},
\end{split}
\end{equation*}
with $\mathcal{D}_{k}$ a sufficiently fine partition of $[t_{k-1},t_{k}]_{\mathbb{T}}$
such that $\mathcal{D}^{0}$ is $\delta_{n}$-fine. Thanks to Henstock's Lemma~\ref{lem1}, we have
$$
\left|\sum_{p(\xi_{k})\geq n}f_{n}(\xi_{k})(t_{k}-t_{k-1})
-\sum_{p(\xi_{k})\geq n}\int_{t_{k-1}}^{t_{k}}f_{n}(t)\Delta t\right|
\leq\bigvee_{i=1}^{\infty}a_{n,i,\varphi(i+n+1)}.
$$
Consequently, we obtain
\begin{equation*}
\begin{split}
\Bigg|S&(f_{n},\mathcal{D}_{n},\delta)-\int_{a}^{b}f_{n}(t)\Delta t\Bigg|\\
&\leq\left|\sum_{p(\xi_{k})\geq n}f_{n}(\xi_{k})(t_{k}-t_{k-1})
-\sum_{p(\xi_{k})\geq n}\int_{t_{k-1}}^{t_{k}}f_{n}(t)\Delta t\right|\\
&\quad +\left|\sum_{p(\xi_{k})< n}f_{n}(\xi_{k})(t_{k}-t_{k-1})
-\sum_{p(\xi_{k})< n}\int_{t_{k-1}}^{t_{k}}f_{n}(t)\Delta t\right|\\
&\leq\bigvee_{i=1}^{\infty}a_{n,i,\varphi(i+n+1)}
+\left|\sum_{m=1}^{n-1}\sum_{p(\xi_{k})=m}f_{n}(\xi_{k})(t_{k}-t_{k-1})
-\sum_{m=1}^{n-1}\sum_{p(\xi_{k})=m}\int_{t_{k-1}}^{t_{k}}f_{n}(t)\Delta t\right|\\
&\leq \bigvee_{i=1}^{\infty}a_{n,i,\varphi(i+n+1)}
+\left|\sum_{m=1}^{n-1}\sum_{p(\xi_{k})=m}f_{n}(\xi_{k})(t_{k}-t_{k-1})
-\sum_{m=1}^{n-1}\sum_{p(\xi_{k})=m}\int_{t_{k-1}}^{t_{k}}f_{\xi_{k}}(t)\Delta t\right|\\
&\leq \bigvee_{i=1}^{\infty}a_{n,i,\varphi(i+n+1)}
+\left|\sum_{m=1}^{n-1}\sum_{p(\xi_{k})=m}f_{n}(\xi_{k})(t_{k}-t_{k-1})
-\sum_{m=1}^{n-1}\sum_{p(\xi_{k})=m}f_{p(\xi_{k})}(\xi_{k})(t_{k}-t_{k-1})\right|\\
&\quad +\left|\sum_{m=1}^{n-1}\sum_{p(\xi_{k})=m}
f_{p(\xi_{k})}(\xi_{k})(t_{k}-t_{k-1})-\sum_{m=1}^{n-1}
\sum_{p(\xi_{k})=m}\int_{t_{k-1}}^{t_{k}}f_{\xi_{k}}(t)\Delta t\right|\\
&\leq \bigvee_{i=1}^{\infty}a_{n,i,\varphi(i+n+1)}
+\sum_{m=1}^{n-1}\sum_{p(\xi_{k})=m}
\left|f_{n}(\xi_{k})-f_{p(\xi_{k})}(\xi_{k})\right|(t_{k}-t_{k-1})\\
&\quad +\sum_{m=1}^{n-1}\left|\sum_{p(\xi_{k})=m}f_{p(\xi_{k})}(\xi_{k})(t_{k}-t_{k-1})
-\sum_{p(\xi_{k})=m}\int_{t_{k-1}}^{t_{k}}f_{\xi_{k}}(t)\Delta t\right|\\
&<\bigvee_{i=1}^{\infty}a_{n,i,\varphi(i+n+1)}
+2(b-a)\bigvee_{i=1}^{\infty}a_{i,\varphi(i+1)}
+\sum_{m=1}^{n-1}\bigvee_{i=1}^{\infty}a_{m,i,\varphi(i+m+1)}\\
&<\bigvee_{i=1}^{\infty}b_{1,i,\varphi(i+n+1)}
+\sum_{m=1}^{n-1}\bigvee_{i=1}^{\infty}a_{m,i,\varphi(i+m+1)}\\
&<\bigvee_{i=1}^{\infty}b_{1,i,\varphi(i+n+1)}
+\sum_{m=2}^{n-1}\bigvee_{i=1}^{\infty}b_{m,i,\varphi(i+m+1)}\\
&<\sum_{m=1}^{\infty}\bigvee_{i=1}^{\infty}b_{m,i,\varphi(i+m)}.
\end{split}
\end{equation*}
On the other hand, we have
$$
\left|S(f_{n},\mathcal{D},\delta)
-\int_{a}^{b}f_{n}(t)\Delta t\right|
\leq (L-l)(b-a)=x.
$$
Then,
$$
\left|S(f_{n},\mathcal{D},\delta)-\int_{a}^{b}f_{n}(t)\Delta t\right|
\leq x\bigwedge\Big(\sum_{m=1}^{\infty}\bigvee_{i=1}^{\infty}
b_{m,i,\varphi(i+m)}\Big)\leq\bigvee_{i=1}^{\infty}b_{i,\varphi(i)}.
$$
Now, we prove that $\{f_{n}\}_{n=1}^{\infty}$ is a sequence of uniformly
HK $\Delta$-integrable functions on $[a,b]_{\mathbb{T}}$. By Theorem~\ref{th5},
$f$ is HK $\Delta$-integrable and
$$
\lim_{n\rightarrow\infty}\int_{a}^{b}f_{n}(t)\Delta t
=\int_{a}^{b}f(t)\Delta t.
$$
The proof is complete.
\end{proof}


\section*{Acknowledgements}

This research is supported by Educational Commission
of Hubei Province, grant no. B2016160 (You and Zhao)
and by FCT and CIDMA within project UID/MAT/04106/2013 (Torres). 
The authors are grateful to one anonymous referee 
for valuable comments and suggestions.



\end{document}